\documentclass[11pt,a4paper]{scrartcl}

\usepackage[utf8]{inputenc}
\usepackage{microtype}
\usepackage{amsmath}
\usepackage{amsfonts}
\usepackage{amssymb}
\usepackage{amsthm}
\usepackage{enumitem}
\usepackage{lmodern}
\usepackage{microtype}
\usepackage[]{scrlayer-scrpage}
\usepackage[english]{babel}
\usepackage{xcolor}
\usepackage{xspace}		
\usepackage{xpatch} 
\usepackage{aliascnt} 

\usepackage[
style=numeric,    		
isbn=false,                
doi=false,
url=false,
pagetracker=true,          
maxbibnames=5,            
maxcitenames=3,            
giveninits=true,			
autocite=inline,           
block=space,               
date=comp,                
dateabbrev=false,			
backend=biber,
]{biblatex}
\DeclareFieldFormat[article,incollection,inproceedings,thesis]{citetitle}{#1\midsentence} 
\DeclareFieldFormat[article,incollection,inproceedings,thesis]{title}{#1\midsentence}
\DeclareFieldFormat{titlecase}{\MakeSentenceCase*{#1}} 
\DeclareFieldFormat{titlecase}{\MakeTitleCase{#1}}
\newrobustcmd{\MakeTitleCase}[1]{%
	\ifthenelse{\ifcurrentfield{booktitle}\OR\ifcurrentfield{booksubtitle}%
		\OR\ifcurrentfield{maintitle}\OR\ifcurrentfield{mainsubtitle}%
		\OR\ifcurrentfield{journaltitle}\OR\ifcurrentfield{journalsubtitle}%
		\OR\ifcurrentfield{issuetitle}\OR\ifcurrentfield{issuesubtitle}%
		\OR\ifentrytype{book}\OR\ifentrytype{mvbook}\OR\ifentrytype{bookinbook}%
		\OR\ifentrytype{booklet}\OR\ifentrytype{suppbook}%
		\OR\ifentrytype{collection}\OR\ifentrytype{mvcollection}%
		\OR\ifentrytype{suppcollection}\OR\ifentrytype{manual}%
		\OR\ifentrytype{periodical}\OR\ifentrytype{suppperiodical}%
		\OR\ifentrytype{proceedings}\OR\ifentrytype{mvproceedings}%
		\OR\ifentrytype{reference}\OR\ifentrytype{mvreference}%
		\OR\ifentrytype{report}\OR\ifentrytype{thesis}}
	{#1}
	{\MakeSentenceCase{#1}}}
\renewrobustcmd*{\bibinitdelim}{\,} 
\xpatchbibdriver{online}{\usebibmacro{date}}{}{}{}
\xpatchbibdriver{online}{ \usebibmacro{addendum+pubstate}}{\addcomma\space\usebibmacro{publisher+location+date}\newunit\newblock\usebibmacro{addendum+pubstate}}{}{} 
\bibliography{../../../Draft/bib.bib}  
\usepackage{url} 
\usepackage[hidelinks]{hyperref} 
\usepackage[capitalise,nameinlink,noabbrev]{cleveref}
\crefformat{equation}{#2(#1)#3}

\theoremstyle{definition}
\newtheorem{definition}{Definition}
\newtheorem{example}{Example}
\newtheorem{remark}{Remark}

\theoremstyle{plain}
\newtheorem{theorem}{Theorem}[section]
\newaliascnt{lemma}{theorem}
\newtheorem{lemma}[lemma]{Lemma}
\aliascntresetthe{lemma}
\newaliascnt{corollary}{theorem}
\newtheorem{corollary}[corollary]{Corollary}
\aliascntresetthe{corollary}
\newaliascnt{proposition}{theorem}
\newtheorem{proposition}[proposition]{Proposition}
\aliascntresetthe{proposition}

\newcommand{\F}{\mathbb{F}}
\newcommand{\Fq}{\mathbb{F}_q}

\newcommand{\GR}{\text{GR}}
\newcommand{\Z}{\mathbb{Z}}

\newcommand{\T}{\mathcal{T}}
\newcommand{\I}{\mathcal{I}}

\begin{document}

\title{On solving isomorphism problems about 2-designs using block intersection numbers}
\author{Christian Kaspers\thanks{Institute for Algebra and Geometry, Otto von Guericke University Magdeburg, 39106 Magdeburg, Germany (email: \href{mailto:christian.kaspers@ovgu.de}{\nolinkurl{christian.kaspers@ovgu.de}}, \href{mailto:alexander.pott@ovgu.de}{\nolinkurl{alexander.pott@ovgu.de}})} \space and Alexander Pott$^*$}
\maketitle

\begin{abstract}
In this paper, we give a partial solution to a new isomorphism problem about $2$-$(v,k,k-1)$ designs from disjoint difference families in finite fields and Galois rings. Our results are obtained by carefully calculating and bounding some block intersection numbers, and we give insight on the limitations of this technique. Moreover, we present results on cyclotomic numbers, the multiplicities of block intersection numbers of certain designs and on the structure of Galois rings of characteristic $p^2$.
\end{abstract}


\section{Introduction}
\label{sec:Introduction}
In their previous work~\cite{kaspers2019}, the present authors studied two constructions of difference families in Galois rings by \textcite{davis2017} and by \textcite{momihara2017}. Both constructions were inspired by a classical construction of difference families in finite fields which was introduced by \textcite{wilson1972} in 1972. Various types of difference families have long been studied in combinatorial literature \cite{abel2007,beth1999,buratti2017,buratti2019,chang2006,furino1991,jedwab2019,wilson1972}. They have applications in coding theory and communications and information security \cite{ng2016}, and they are related to many other combinatorial objects. In particular, every difference family gives rise to a combinatorial design. Combinatorial designs themselves have been extensively studied since the first half of the 19th century, they have many applications in group theory, finite geometry and cryptography \cite{beth1999,colbourn2007}.\par

Whenever a new construction of difference families is given, the natural question arises whether the associated designs are also new or whether they are isomorphic to known designs. By calculating and bounding some block intersection numbers, the present authors~\cite{kaspers2019} solved this isomorphism problem for the difference families from \textcite{momihara2017} and \textcite{wilson1972} and for those from \textcite{davis2017} and \textcite{wilson1972}. In this paper, we obtain new difference families from the ones constructed by \textcite{davis2017}. These new difference families also have an analogue in finite fields from Wilson's \cite{wilson1972} construction. Motivated by the present authors' previous results, we will use the same technique as in their paper~\cite{kaspers2019} to study whether the associated designs are isomorphic or not. It will become clear that the approach to use block intersection numbers as a tool to solve isomorphism problems is promising for certain types of designs but has its limitations in general. \par

We start by defining the objects we study in this paper. First, we need the following notations: Let $G$ be an additively written abelian group, $A,B \subseteq G$ and $g \in G$. We define multisets
\begin{align*}
\Delta A 		&= \{ a - a' : a, a' \in A, a \ne a'\},\\
A-B			 	&= \{a-b : a \in A, b \in B, a \ne b\},\\
A + g 			&= \{a + g : a \in A\}.
\end{align*}
We will sometimes use these notations to denote sets, not multisets. It will be clear from the context whether we mean the multiset or the respective set.

\begin{definition}
	\label{def:DF}
	Let $G$ be an abelian group of order $v$, and let $D_1, D_2, \dots, D_b$ be $k$\nobreakdash-subsets of $G$. The collection $D = \{D_1, D_2, \dots, D_b\}$ is called a \emph{$(v,k,\lambda)$ difference family in $G$} if each nonzero element of $G$ occurs exactly $\lambda$ times in the multiset union
	\[
	\bigcup_{i=1}^b \Delta D_i.
	\]
	If the subsets $D_1, D_2, \dots, D_b$ are mutually disjoint, they form a \emph{disjoint difference family}. If $b = 1$, one speaks of a \emph{$(v,k,\lambda)$ difference set}. We call $D$ \emph{near-complete} if the subsets~$D_1,D_2,\dots,D_b$ partition $G \setminus \{0\}$.
\end{definition}

In this paper, we focus on near-complete $(v,k,k-1)$ disjoint difference families. For more background on this type of difference families, the reader is referred to the survey by \textcite{buratti2017} who summarizes many results and introduces a powerful new construction. His construction includes many known constructions, including the one by \textcite{davis2017}. However, it seems to be too general to use it for studying isomorphism problems, at least when using block intersection numbers. Eventually, we remark that every near-complete disjoint difference family is also an external difference family \cite{chang2006,davis2017,kaspers2019}.

As mentioned above, every difference family gives rise to a combinatorial design.
\begin{definition}
	\label{def:design}
	Let $P$ be a set with $v$ elements that are called \emph{points}. A \emph{$t$\nobreakdash-$(v,k,\lambda)$~design}, or \emph{$t$-design}, in brief, is a collection of $k$-subsets, called \emph{blocks}, of $P$ such that every $t$-subset of $P$ is contained in exactly $\lambda$ blocks.
\end{definition}
The associated designs of difference families are $2$-designs which are often referred to as \emph{balanced incomplete block designs (BIBD)}. They are constructed as the development of a difference family.

\begin{definition}
	Let $G$ be an abelian group, and let $D = \{D_1, D_2, \dots, D_b\}$ be a collection of subsets of $G$. The \emph{development of $D$} is the collection
	\[
	dev(D) = \left\lbrace D_i + g : D_i \in D, g \in G\right\rbrace
	\]
	of all the translates of the sets $D_1,D_2,\dots,D_b$. The sets $D_1, D_2, \dots, D_b$ are called the \emph{base blocks} of $dev(D)$.
\end{definition}

In other words, the development of $D$ is the union of the orbits of the sets contained in $D$ under the action of $G$. If all orbits have full length, $dev(D)$ contains $vb$ blocks. The following well-known proposition relates difference families to $2$-designs. 

\begin{proposition}
	\label{prop:dev_design}
	Let $D$ be a $(v,k,\lambda)$ difference family in an abelian group $G$. The development $dev(D)$ of $D$ forms a $2$-$(v,k,\lambda)$ design with point set $G$.
\end{proposition}

\section{Galois rings}
\label{sec:Galois_rings}
In this section, we give a short introduction to Galois rings and present some of their well-known properties needed in this paper. We refer to the work by \textcite{mcdonald1974,wan2003} for extended general background on this topic. Let $p$ be a prime, and let $f(x) \in \Z_{p^m}[x]$ be a monic basic irreducible polynomial of degree~$r\ge1$, which means that the image of $f$ modulo $p$ in $\F_p[x]$ is irreducible. The factor ring 
\[
\Z_{p^m}[x]/ \langle f(x)\rangle
\]
is called a \emph{Galois ring} of characteristic $p^m$ and extension degree $r$. It is denoted by $\GR(p^m, r)$, and its order is $p^{mr}$. Since any two Galois rings of the same characteristic and order are isomorphic, we will speak of \emph{the} Galois ring~$\GR(p^m, r)$.\par
Galois rings are local commutative rings. The unique maximal ideal of the ring $R = \GR(p^m, r)$ is 
\[
\I = pR = \{pa : a \in R\}.
\]
The factor ring $R / \I$ is isomorphic to the finite field $\F_{p^r}$ with $p^r$ elements. As a system of representatives of $R/\I$, we take the \emph{Teichmüller set}
\[
\T = \{0, 1, \xi, \dots, \xi^{p^r-2}\},
\]
where $\xi$ denotes a root of order $p^r-1$ of $f(x)$. It is convenient to choose the \emph{generalized Conway polynomial}, that is the Hensel lift from $\F_p[x]$ to $\mathbb{Z}_{p^m}[x]$ of the Conway polynomial, as our polynomial $f(x)$. Then, $x +\langle f \rangle$ is a generator of the Teichmüller group, and we set $\xi = x + \langle f \rangle$. \textcite[Section~1.3]{zwanzger2011} provides more information on the generalized Conway polynomial and its construction. Every $a\in R$ has a unique \emph{$p$-adic representation} $a = \alpha_0 + p\alpha_1 + \dots + p^{m-1}\alpha_{m-1}$, where $\alpha_0, \alpha_1, \dots, \alpha_{m-1} \in \T$.\par
The elements of $R \setminus \I$ are all the units of $R$. We denote this \emph{unit group} by $R^*$. It has order $p^{mr} - p^{(m-1)r}$ and is the direct product of the cyclic \emph{Teichmüller group}
\[
\T^* = \T \setminus \{0\}
\] of order $p^r-1$ and the \emph{group of principal units $\mathbb{P} = 1 + \I$} of order $p^{(m-1)r}$. If $p$ is odd or if $p = 2$ and $m \leq 2$, then $\mathbb{P}$ is a direct product of $r$ cyclic groups of order $p^{m-1}$. If $p = 2$ and $m \geq 3$, then $\mathbb{P}$ is a direct product of a cyclic group of order $2$, a cyclic group of order $2^{m-2}$, and $r-1$ cyclic groups of order~$2^{m-1}$. In this paper, we will only consider Galois rings of characteristic~$p^2$. In this case, $(1+p\alpha)(1+p\beta) = 1+p(\alpha + \beta)$ for any $\alpha, \beta \in \T$, and every unit $u \in \GR(p^2, r)^*$ has a unique representation 
\[
u = \alpha_0 (1 + p\alpha_1),
\]
where $\alpha_0 \in \T^*$ and $\alpha_1 \in \T$. Moreover, the group of principal units $\mathbb{P}$ is a direct product of $r$ cyclic groups of order $p$ and thus has the structure of an elementary abelian group of order $p^r$.

\section{Construction of disjoint difference families}
In this section, we describe three constructions of disjoint difference families. The constructions from \autoref{th:Fq_DDF} and \autoref{th:Davis_EDF} are well known. The third construction, in \autoref{th:squares_GR_DDF}, follows from results by \textcite{furino1991}. As the first two constructions also fall into \citeauthor{furino1991}'s very general framework, we will present his result first. We only restate a special case of his construction.

\begin{theorem}[{\cite[Theorem~3.3 and Corollary~3.5]{furino1991}}]
	\label{th:Furino}
	Let $R$ be a commutative ring with an identity. Denote the cardinality of $R$ by $v$ and the unit group of $R$ by~$R^*$. Let $B$ be a subgroup of~$R^*$ of order $k$ such that $\Delta B$ is a subset of $R^*$. Denote by~$S$ a system of representatives of the cosets of $B$ in $R \setminus \{0\}$. The collection~$\{sB : s \in S\}$ is a $(v,k,k-1)$~disjoint difference family in the additive group of $R$.
\end{theorem}

Note that, by abuse of denotation, we also call sets $sB$ where $s$ is not a unit a coset of $B$. Because of the condition $\Delta B \subseteq R^*$, these cosets also have cardinality~$k$, and all the cosets partition $R \setminus \{0\}$. 

Next, we present the construction of disjoint difference families in finite fields by \textcite{wilson1972}. It makes use of the cyclotomoy of the $e$-th powers in a finite field. 
\begin{theorem}
	\label{th:Fq_DDF}
	Let $\F_q$ be the finite field with $q$ elements, and let $\alpha$ be a generator of the multiplicative group $\Fq^*$ of $\Fq$. Moreover, let $e,f$ be integers satisfying $ef = q-1$, where $e,f \geq 2$, and let
	\[
	C_i = \{ \alpha^{t}\ : \ t \equiv i \pmod e  \},
	\]
	where $i = 0, 1, \dots, e-1$, be the cosets of the unique subgroup $C_0$ of index $e$ and order $f$ that is formed by the $e$-th powers of~$\alpha$ in $\Fq^*$. Then, the collection $C = \{C_0, C_1, \dots, C_{e-1}\}$ is a near-complete $(q, f, f-1)$ disjoint difference family in the additive group of $\F_q$.
\end{theorem}

We now present the construction of disjoint difference families by \textcite{davis2017}. We use the same notation as in \autoref{sec:Galois_rings}, and we remark that this theorem also follows from \autoref{th:Furino} and from a result by \textcite{buratti2017}.

\begin{theorem}[{\cite[Theorem~4.1]{davis2017}}]
	\label{th:Davis_EDF}
	Let $p$ be a prime, and let $r$ be a positive integer such that $p^r \ge 3$. Denote by $\T$ the Teichmüller set of the Galois ring~$\GR(p^2,r)$ and by $\T^*$ the Teichmüller group $\T^* = \T \setminus \{0\}$. The collection
	\[
	E = \left \lbrace (1+p\alpha)\T^* : \alpha \in \T \right \rbrace \cup p\T^*
	\]
	forms a near-complete $(p^{2r}, p^r-1, p^r-2)$ disjoint difference family in the additive group of $\GR(p^2, r)$.
\end{theorem}

Since $p^r-1$ divides $p^{2r}-1$, there exists a disjoint difference family in the additive group of $\F_{p^{2r}}$ that has the exact same parameters as the difference family from \autoref{th:Davis_EDF}. It can be constructed using \autoref{th:Fq_DDF} by taking the $(p^r+1)$-th powers in $\F_{p^{2r}}$. Inspired by \autoref{th:Davis_EDF} and the work by \textcite{furino1991}, we noticed that if $p$ is odd, we obtain a new disjoint difference family by taking the cosets of the group of \emph{Teichmüller squares}. 
\begin{theorem}
	\label{th:squares_GR_DDF}
	Let $p$ be an odd prime and let $r$ be a positive integer such that $p^r \ge 5$. Moreover, let 
	\[
		\T^* = \{1,\xi,\xi^2, \dots, \xi^{p^r-2}\}
	\] 
	be the Teichmüller group of the Galois ring $\GR(p^2,r)$, and let $\T = \T^*\cup\{0\}$. By 
	\[
		\T_S^* = \{1, \xi^2, \dots, \xi^{p^r-3}\}
	\]
	we denote the set of squares and by 
	\[
		\T_N^* = \{\xi, \xi^3, \dots, \xi^{p^r-2}\}
	\]
	we denote the set of non-squares in $\T^*$. The collection
	\[
		E^H = \{(1+p\alpha)\T_S^*:\alpha \in \T\} \cup \{p\T_S^*\} \cup \{(1+p\alpha)\T_N^* : \alpha \in \T\} \cup \{p\T_N^*\}
	\]
	forms a near complete $\left(p^{2r}, \frac{p^r-1}{2}, \frac{p^r-3}{2}\right)$ disjoint difference family in the additive group of $\GR(p^2,r)$.
\end{theorem}
\begin{proof}
	Denote by $\I = p\GR(p^2,r)$ the maximal ideal of $\GR(p^2,r)$. The Teichmüller set~$\T$ is a system of representatives of $\GR(p^2,r)/\I$. This factor ring is isomorphic to the finite field $\F_{p^{r}}$. Consequently, the difference of two distinct elements of the Teichmüller group~$\T^*$ is a unit, hence $\Delta\T^* \subseteq \GR(p^2,r)^*$. As $\T^*_S$ is a subgroup of~$\T^*$, it follows that $\Delta \T^*_S$ is a subset of the unit group $\GR(p^2,r)^*$.  In this case, according to \autoref{th:Furino}, the collection of the cosets of $\T^*_S$ in $\GR(p^2,r) \setminus \{0\}$ forms a disjoint difference family in the additive group of $\GR(p^2,r)$.
\end{proof}
Note that the difference family $E^H$ from \autoref{th:squares_GR_DDF} can be obtained from the difference family $E$ presented in \autoref{th:Davis_EDF} by cutting the base blocks of $E$ into halves, hence the name~$E^H$. Furthermore, note that there exists a difference family $C^H$ in the finite field $\F_{p^{2r}}$ which has the same parameters as~$E^H$. According to \autoref{th:Fq_DDF}, the cosets of the subgroup~$C^H_0$ of the $2(p^r+1)$-th powers in $\F_{p^{2r}}^*$ form a $\left(p^{2r}, \frac{p^r-1}{2}, \frac{p^r-3}{2}\right)$ disjoint difference family in the additive group of $\F_{p^{2r}}$. In the following section, we will study the isomorphism problem for the difference families $C^H$ and $E^H$ from finite fields and Galois rings. Moreover, we present additional isomorphism invariants of the designs in finite fields coming from \autoref{th:Fq_DDF}.

\section{A partial solution to the isomorphism problem}
\label{sec:isoproblem}
Denote by $C$ the $(p^{2r},p^r-1,p^r-2)$ difference family and by $C^H$ the $\left(p^{2r}, \frac{p^r-1}{2}, \frac{p^r-3}{2}\right)$ difference family in the additive group of $\F_{p^{2r}}$ which are constructed using \autoref{th:Fq_DDF}. Denote by $E$ the $(p^{2r},p^r-1,p^r-2)$ difference family and by $E^H$ the $\left(p^{2r}, \frac{p^r-1}{2}, \frac{p^r-3}{2}\right)$ difference family in the additive group of $\GR(p^2,r)$ which are constructed using \autoref{th:Davis_EDF} and \autoref{th:squares_GR_DDF}, respectively.\par

In their previous work, the present authors~\cite{kaspers2019} solved the isomorphism problem for the $2$-$(p^{2r},p^r-1,p^r-2)$ designs $dev(C)$ and $dev(E)$. They showed that the designs are nonisomorphic for all combinations of $p$ and $r$ except $p=3$ and $r=1$. In this section, we will give a partial solution to the isomorphism problem for the $2$-$\left(p^{2r}, \frac{p^r-1}{2}, \frac{p^r-3}{2}\right)$ designs $dev(C^H)$ and $dev(E^H)$. Note that these designs can be obtained from $dev(C)$ and $dev(E)$, respectively, by cutting every block into two halves.

\begin{remark}
	\label{rem:feng}
	The fact that two designs $\mathcal{D}_1, \mathcal{D}_2$ are nonisomorphic does not imply that two designs $\mathcal{D}_1^H,\mathcal{D}_2^H$ that are obtained by cutting the blocks of $\mathcal{D}_1$ and $\mathcal{D}_2$ into smaller blocks are nonisomorphic. This is shown in the following example which was given in the context of skew Hadamard difference sets by \textcite[Example~3.3]{feng2012}. Denote by $C_0$ the subgroup of the $14$\nobreakdash-th~powers of the multiplicative group of the finite field~$\F_{11^3}$ and by $C_0, C_1, \dots, C_{13}$ the cosets of~$C_0$. It follows from \autoref{th:Fq_DDF} that the collection $C = \{C_0,C_1,\dots, C_{13}\}$ is a disjoint difference family in the additive group of $\F_{11^3}$. The collections
	\begin{align*}
	D_1 &= \{\{C_0 \cup C_2 \cup C_4 \cup C_6 \cup C_8 \cup C_{10} \cup C_{12}\},\\ &\qquad\{C_1 \cup C_3 \cup C_5 \cup C_7 \cup C_9 \cup C_{11} \cup C_{13}\}\},\\
	D_2 &= \{\{C_0 \cup C_1 \cup C_2 \cup C_3 \cup C_4 \cup C_5 \cup C_6 \},\\ &\qquad \{C_7 \cup C_8 \cup C_9 \cup C_{10} \cup C_{11} \cup C_{12} \cup C_{13}\}\},\\
	D_3 &= \{\{C_0 \cup C_1 \cup C_3 \cup C_4 \cup C_5 \cup C_6 \cup C_{9}\},\\ &\qquad \{C_2 \cup C_7 \cup C_8 \cup C_{10} \cup C_{11} \cup C_{12} \cup C_{13}\}\}.
	\end{align*}
	are also disjoint difference families in the additive group of $\F_{11^3}$. Consider their associated designs $dev(D_1), dev(D_2), dev(D_3)$. Their full automorphism groups~$\mathcal{A}_1, \mathcal{A}_2,\mathcal{A}_3$ have orders $|\mathcal{A}_1| = 5310690, |\mathcal{A}_2| = 252890$ and $|\mathcal{A}_3| = 758670$. Thus, the designs are pairwise nonisomorphic. However, it is clear that from all three difference families, we can obtain the difference family $C$ by cutting their base blocks into the cyclotomic cosets $C_0, C_1,\dots,C_{13}$. Hence, from the nonisomorphic designs $dev(D_1),dev(D_2),dev(D_3)$, we can obtain the exact same design $dev(C)$.
\end{remark}

The present authors~\cite{kaspers2019} solved the isomorphism problem for $dev(C)$ and $dev(E)$ by comparing the block intersection numbers of these designs.
\begin{definition}
	\label{def:blockintersectionnumber}
	We call an integer $N$ a \emph{block intersection number} of a $t$\nobreakdash-design~$\mathcal{D}$, if $\mathcal{D}$ contains two distinct blocks $B$ and $B'$ that intersect in $N$~elements.
\end{definition}
Block intersection numbers are invariant under isomorphism. For a given design~$\mathcal{D}$, they can be easily computed as the entries of the matrix $M^TM$, where $M$ is the incidence matrix of $\mathcal{D}$ with the rows corresponding to the points and the columns corresponding to the blocks of~$\mathcal{D}$. Note, however, that there also exist designs that have the exact same block intersection numbers but are nonisomorphic. One example are the designs given in \autoref{rem:feng}. These designs are pairwise nonisomorphic, but they all share the intersection numbers $0,332,333$. For the designs $dev(C^H)$ and $dev(E^H)$, however, block intersection numbers seem to distinguish the designs as the following example shows.
\begin{example}
\label{ex:easy_example}
	The constructions from \autoref{th:Fq_DDF} and \autoref{th:squares_GR_DDF} yield $(625,12,11)$~disjoint difference families $C^H$ and $E^H$ in the additive groups of~$\F_{5^4}$ and $\GR(25,2)$, respectively. The associated $2$-$(625,12,11)$~designs have the following block intersection numbers: for $dev(C^H)$, they are $0,1,5,6$, and for $dev(E^H)$, they are $0,1,2,5,6$. Hence, the two designs are nonisomorphic.
\end{example}

Before we start with the actual calculation of our block intersection numbers we focus on their multiplicities.

\begin{remark}
\label{rem:multiplicities_BIN}
	Not only the block intersection numbers themselves but also their multiplicities are isomorphism invariants of a combinatorial design. Hence, in the following, we will not only state the intersection numbers but also their multiplicities whenever it is possible. Although we will not use the multiplicities to solve an isomorphism problem in this paper, they might be useful for further research. To determine the multiplicity of an intersection number $N$, we will first count the number of pairs $(i,j)$ such that two blocks $B_i$ and $B_j$ intersect in $N$ elements without considering that $B_i \cap B_j = B_j \cap B_i$. In the end, we divide this number by $2$.
\end{remark}

\begin{example}
	\label{ex:multiplicities}
	The multiplicities of the block intersection numbers of the designs from \autoref{ex:easy_example} are as follows: In $dev(C^H)$, the intersection numbers $0,1,5,6$ occur with multiplicities $410\,328\,750$, $117\,000\,000$, $195\,000$, and $585\,000$, respectively. In $dev(D^H)$, the intersection numbers $0,1,2,5,6$ have multiplicities $417\,078\,750$, $100\,687\,500$, $10\,312\,500$, $7\,500$ and $22\,500$, respectively. Hence, the multiplicities distinguish the designs.\par
	For all three designs given in \autoref{rem:feng}, the multiplicities of the block intersection numbers $0,332,333$ are $1331$, $2\,655\,345$ and $885\,115$, respectively. Hence, in this case, neither the intersection numbers nor their multiplicities distinguish the designs.
\end{example}

In the remainder of this section, we will first generally describe the block intersection numbers and their multiplicities of the designs coming from the disjoint difference families in~$\F_q$ that we presented in \autoref{th:Fq_DDF}. From this result, we will derive the block intersection numbers and their multiplicities of the designs $dev(C)$ and $dev(C^H)$. Note that the present authors, in their previous paper~\cite{kaspers2019}, already gave the intersection numbers of $dev(C)$. We will now contribute the associated multiplicities. Finally, we will establish bounds on the intersection numbers of $dev(E^H)$ which will lead to a partial solution of the isomorphism problem of the designs $dev(C^H)$ and $dev(E^H)$.

The block intersection numbers of the design from \autoref{th:Fq_DDF} are strongly related to the so-called cyclotomic numbers: As in \autoref{th:Fq_DDF}, let $C_0, C_1, \dots, C_{e-1}$ be the cosets of the subgroup~$C_0$ of the $e$-th powers in $\Fq^*$. For fixed non-negative integers $i,j \le e-1$, the \emph{cyclotomic number $(i,j)_e$ of order $e$} is defined as
\[
	(i,j)_e = |(C_i + 1) \cap C_j|.
\]
Denote by $n_e(N)$ the number of pairs $(i,j)$, where $i,j \le e-1$, such that the cyclotomic number $(i,j)_e = N$.

\begin{proposition}
	\label{prop:cyclo_BIN}
	Let $e,f \ge 2$ be integers such that $ef = q-1$, and let $C$ be a $(q,f,f-1)$ disjoint difference family in the additive group of $\F_q$ constructed with \autoref{th:Fq_DDF}. The block intersection numbers of the $2$-$(q,f,f-1)$ design~$dev(C)$ are $0$ and the values of the cyclotomic numbers $(i,j)_e$ of order $e$ in $\F_q$. The intersection number~$0$ has multiplicity~$\frac{1}{2}q(q-1)n_e(0) + \frac{1}{2}qe(e-1)$, each nonzero intersection number $N$ has multiplicity $\frac{1}{2}q(q-1)n_e(N)$.
\end{proposition}
\begin{proof}
	Denote by $\alpha$ a primitive element of the finite field $\F_q$. Let $C = \{C_0,C_1, \dots, C_{e-1}\}$ be a disjoint difference family from \autoref{th:Fq_DDF} in the additive group of $\F_q$. Take two arbitrary distinct blocks $C_i + a$ and $C_j + b$ of $dev(C)$. If we want to calculate the cardinality 
	\[
	|(C_i + a) \cap (C_j + b)|
	\]
	of their intersection, we need to determine the number of solutions $(s,t)$ of the equation
	\begin{align}
	\label{eq:cyclo}
		\alpha^{se + i} + a  =  \alpha^{te+j} + b.
	\end{align}
	If $a=b$, then obviously only the case $i \ne j$ is relevant. As $C_i$ and $C_j$ are disjoint, there are no solutions in this case and
	\[
		|(C_i +a) \cap (C_j + a)| = 0.
	\]
	Since there are~$q$ choices for $a$ and $e(e-1)$ choices for $(i,j)$ such that $i \ne j$, the block intersection number $0$ occurs $qe(e-1)$ times in this context. Removing repeated intersections, this multiplicity reduces to $\frac{qe(e-1)}{2}$.\par
	
	If $a \ne b$, then $a-b = \alpha^r$ for some $r \in \{0,\dots,q-1\}$. Write $r = me+r'$ such that $0 \le r' \le e-1$. Now, we can rewrite \cref{eq:cyclo} as
	\[
	\alpha^{(s-m)e+(i-r')}  + 1= \alpha^{(t-m)e + (j-r')}.
	\]
	Consequently,
	\[
	|(C_i + a) \cap (C_j + b)| = |(C_{i-r'}+1) \cap C_{j-r'}|,
	\]
	where the subscripts are calculated modulo $e$. The right-hand side of the above equation is exactly the cyclotomic number $(i-r',j-r')_e$. We have $q(q-1)$ choices for $(a,b)$ such that $a \ne b$, and the difference $a-b$ covers all the elements of $\F_q^*$ the same number of times. Consequently, each cyclotomic number $(i,j)_e$ that equals $N$ contributes with $q(q-1)$ to the multiplicity of the block intersection number $N$. Removing repeated intersections, this contribution reduces to $\frac{q(q-1)}{2}$.
\end{proof}

Using a result by \textcite[Theorems~2 and 4]{baumert1982}, the present authors~\cite{kaspers2019} showed that the cyclotomic numbers of order $p^r+1$ in $\F_{p^{2r}}$ are given as
\begin{align}
\label{eq:cyclonums_p^r+1}
	(0,0)_{p^r+1} &= p^r-2, \nonumber\\
	(0,i)_{p^r+1} = (i,0)_{p^r+1} = (i,i)_{p^r+1} &= 0 &&\textrm{for } i \ne 0,\\
	(i,j)_{p^r+1} &= 1 &&\textrm{for } i \ne j \textrm{ and } i,j \ne 0. \nonumber
\end{align}
With the help of \autoref{prop:cyclo_BIN} we can now determine the block intersection numbers of the $2$-$(p^{2r},p^r-1,p^r-2)$ design $dev(C)$ and their multiplicities. While the present authors~\cite{kaspers2019} presented these block intersection numbers before, the results about their multiplicities are new.

\begin{corollary}
	\label{cor:intnums_mult_1}
	The $2$-$(p^{2r},p^r-1,p^r-2)$ design $dev(C)$ has exactly the following block intersection numbers.
	\[
	\begin{array}{ll}
		\text{block intersection number}	&\text{multiplicity}\\\hline
		0		& \frac{1}{2}\left(3 p^{5r} + p^{4r} - 2 p^{3r}\right)\rule[.25em]{0em}{1em}\rule[-.75em]{0em}{0em}\\
		1		& \frac{1}{2}\left(p^{6r} - p^{5r} - p^{4r} + p^{3r}\right)\rule[-.75em]{0em}{0em}\\
		p^r-2	& \frac{1}{2}\left(p^{4r}-p^{2r}\right)
	\end{array}
	\]
\end{corollary}
\begin{proof}
	Let $e = p^r+1$. Denote by $n_e(N)$ the number of cyclotomic numbers of order~$e$ that equal $N$. In $\F_{p^{2r}}$, according to \cref{eq:cyclonums_p^r+1}, we have
	\begin{align}
	\label{eq:Ne}
		n_e(0) = 3p^r,&& n_e(1) = p^r(p^r-1) &&\text{and}&& n_e(p^r-2) = 1.
	\end{align}
	We multiply these numbers with the factor $\frac{1}{2}p^{2r}({p^{2r}}-1)$ from \autoref{prop:cyclo_BIN}. This gives us the multiplicities of the block intersection numbers $1$ and $p^r-2$. To obtain the multiplicity for $0$, according to \autoref{prop:cyclo_BIN}, we additionally need to add  $\frac{1}{2}p^{2r}(p^r+1)p^r$.
\end{proof}

Next, we determine the cyclotomic numbers of order $2(p^r+1)$ in $\F_{p^{2r}}$ which are the intersection numbers of $dev(C^H)$. Unfortunately, these parameters no longer match the conditions of the theorems by \textcite{baumert1982} that were used to obtain the cyclotomic numbers of order $p^r+1$. Nevertheless, we can deduce these cyclotomic numbers from~\cref{eq:cyclonums_p^r+1} with the help of the following well-known lemma.

\begin{lemma}[{\cite[§67]{dickson1958}, \cite[Theorem~2]{ralston1979}}]
	\label{lem:squares+1}
	Let $p$ be an odd prime. Let $S$ be the set of nonzero squares and $N$ be the set of non-squares in the finite field $\F_{p^{r}}$. Denote by~$QQ$ the number of squares $s \in S$ for which $s+1$ is a nonzero square and by $QN$ the number of $s \in S$ for which $s+1$ is not a square. Moreover, let $NN$ denote the number of non-squares~$n \in N$ for which $n+1$ is not a square and $NQ$ the number of $n\in N$ for which $n+1$ is a nonzero square.
	\begin{itemize}
		\item If $p^r-1 \equiv 0 \pmod{4}$, then
		\begin{align*}
			QQ &= \frac{p^r-5}{4},	& QN &= \frac{p^r-1}{4}, & NN &= \frac{p^r-1}{4},	& NQ &= \frac{p^r-1}{4}.
		\end{align*}
		\item If $p^r-1 \equiv 2 \pmod{4}$, then
		\begin{align*}
			QQ &= \frac{p^r-3}{4},	& QN &= \frac{p^r+1}{4}, & NN &= \frac{p^r-3}{4},	& NQ &= \frac{p^r-3}{4}.
		\end{align*}
	\end{itemize}
\end{lemma}

Combining \autoref{lem:squares+1} with \cref{eq:cyclonums_p^r+1}, we obtain the following result.

\begin{proposition}
	\label{prop:cyclotomic_numbers_2(p^r+1)}
	Let $p$ be an odd prime, and let $e = p^r+1$ for some positive integer $r$. In $\F_{p^{2r}}$, the cyclotomic numbers of order $2e$ are as follows: 
	\begin{itemize}
		\item If $p^r-1 \equiv 0 \pmod{4}$, then
		\begin{align*}
			(0,0)_{2e} &= \frac{p^r-5}{4},\\
			(0,e)_{2e} = (e,0)_{2e} = (e,e)_{2e} &= \frac{p^r-1}{4}.
		\end{align*}
		\item If $p^r-1 \equiv 2 \pmod{4}$, then
		\begin{align*}
			(0,e)_{2e} &= \frac{p^r+1}{4},\\
			(0,0)_{2e} = (e,0)_{2e} = (e,e)_{2e} &= \frac{p^r-3}{4}.
		\end{align*}
	\end{itemize}
	In both of the above cases, 
	\begin{align*}
		(0,i)_{2e} &= (i,0)_{2e} = (i,i)_{2e} = (i,e)_{2e}\\
		&= (e,i)_{2e} = (i,e+i)_{2e} = (e+i,i)_{2e} = 0 &\text{for } i \notin \{0,e\}.
	\end{align*}
	Out of the remaining cyclotomic numbers, 
	\begin{align*}
		(i,j)_{2e}, (i,j+e)_{2e}, (i+e,j)_{2e}, (i+e,j+e)_{2e}, &&\mbox{where }i,j\ne 0 \mbox{ and } i \ne j,
	\end{align*}
	for each choice of $i$ and $j$, exactly one cyclotomic number is $1$ and the other three cyclotomic numbers are $0$, but it is not known which one is $1$.
\end{proposition}
\begin{proof}
	Let $\alpha$ be a generator of $\F_{p^{2r}}^*$, let $C_0$ be the unique subgroup of order $p^r-1$ of $\F_{p^{2r}}^*$ formed by the $(p^r+1)$-th powers, and let $C_0, C_1, \dots, C_{p^r}$ be the cosets of $C_0$. The finite field $\F_{p^{2r}}$ contains a unique subfield~$\F_{p^r}$ with $p^r$ elements. Hence, the group $C_0$ is the multiplicative group $\F_{p^r}^*$ of the subfield $\F_{p^r}$. As $p^r$ is odd, $C_0$ consists of $\frac{1}{2}(p^r-1)$~squares and non-squares in $\F_{p^{r}}$ each. Consequently, 
	\[
		C_0 = C_0^H \cup C^H_{e},
	\]
	where 
	\[
		C_0^H = \{\alpha^t\ |\ t \equiv 0 \pmod{2(p^r+1)}\}
	\]
	is the set of squares and 
	\[
		C^H_{e} = \{\alpha^t\ |\ t \equiv e \pmod{2(p^r+1)}\}
	\]
	is the set of non-squares in~$\F_{p^{r}}^*$. The values of the cyclotomic numbers $(i,j)_{2e}$, where $i,j \in \{0,e\}$, now follow from \autoref{lem:squares+1}. In the same way as before, we can divide each of the cosets~$C_0,C_1,\dots,C_{p^r}$, of~$C_0$ into two cosets $C_i^H$ and $C_{e+i}^H$ of $C_0^H$. Since 
	\[
		C_i = C_i^H \cup C^H_{e+i}
	\]
	for all $i = 0,1,\dots,p^r$, we obtain
	\[
		(C_i+1) \cap C_j
		= \bigcup_{\substack{k \in \{i,e+i\} \\ \ell \in \{j,e+j\}}} \left(C_k^H+1\right) \cap C_\ell^H
	\]
	for $0 \le i,j \le p^r$. In terms of cyclotomic numbers, this means
	\begin{align}
	\label{eq:cyclotomic_numbers_sum}
	(i,j)_e = \sum_{\substack{k \in \{i,e+i\} \\ \ell \in \{j,e+j\}}} (k,\ell)_{2e}
	\end{align}
	for $0 \le i,j \le p^r$. The values of the cyclotomic numbers $(i,j)_{2e}$, where $i,j \notin \{0,e\}$, now follow from combining \cref{eq:cyclotomic_numbers_sum} with \cref{eq:cyclonums_p^r+1}.  
\end{proof}
Unfortunately, the exact values of the cyclotomic numbers $(i,j)_{2e}$, $(i,j+e)_{2e}$, $(i+e,j)_{2e}$, $(i+e,j+e)_{2e}$, where $i,j\ne 0$ and $i \ne j$, in $\F_{p^{2r}}$ are not known in general. It is an open problem to determine those. However, \autoref{prop:cyclotomic_numbers_2(p^r+1)} immediately gives us the block intersection numbers of the $2$-design $dev(C^H)$ as well as their multiplicities.

\begin{theorem}
	\label{th:intnums_Fp2r}
	Let $C^H$ be a $\left(p^{2r}, \frac{p^r-1}{2}, \frac{p^r-3}{2}\right)$ difference family in the additive group of~$\F_{p^{2r}}$ constructed using \autoref{th:Fq_DDF}. The associated $2$-$\left(p^{2r}, \frac{p^r-1}{2}, \frac{p^r-3}{2}\right)$ design $dev(C^H)$ has exactly the following block intersection numbers.\\
	If $p^r-1 \equiv 0 \pmod{4}$, then
	\[
		\begin{array}{ll}
			\text{block intersection number}	&\text{multiplicity}\\\hline
			0		& \frac{1}{2}\left(3p^{6r} + 9p^{5r} + p^{4r} - 3p^{3r} + 2p^{2r}\right) \rule[.25em]{0em}{1em}\rule[-.75em]{0em}{0em}\\
			1		& \frac{1}{2}\left(p^{6r} - p^{5r} - p^{4r} + p^{3r}\right) \rule[-.75em]{0em}{0em}\\
			\frac{1}{4}(p^r-5)	& \frac{1}{2}\left(p^{4r}-p^{2r}\right) \rule[-.75em]{0em}{0em}\\
			\frac{1}{4}(p^r-1)	& \frac{1}{2}\left(3p^{4r}-3p^{2r}\right) \rule[-.75em]{0em}{0em}\\
		\end{array}.
	\]
	If $p^r-1 \equiv 2 \pmod{4}$, then
	\[
	\begin{array}{ll}
		\text{block intersection number}	&\text{multiplicity}\\\hline
		0		& \frac{1}{2}\left(3p^{6r} + 9p^{5r} + p^{4r} - 3p^{3r} + 2p^{2r}\right) \rule[.25em]{0em}{1em}\rule[-.75em]{0em}{0em}\\
		1		& \frac{1}{2}\left(p^{6r} - p^{5r} - p^{4r} + p^{3r}\right) \rule[-.75em]{0em}{0em}\\
		\frac{1}{4}(p^r-3)	& \frac{1}{2}\left(3p^{4r}-3p^{2r}\right) \rule[-.75em]{0em}{0em}\\
		\frac{1}{4}(p^r+1)	& \frac{1}{2}\left(p^{4r}-p^{2r}\right) \rule[-.75em]{0em}{0em}\\
	\end{array}.
	\]
\end{theorem}
\begin{proof}
	Let $e = p^r+1$. It follows from \autoref{prop:cyclo_BIN} that the block intersection numbers of $dev(C^H)$ are exactly $0$ and the cyclotomic numbers from \autoref{prop:cyclotomic_numbers_2(p^r+1)}. We obtain their multiplicities using \cref{eq:cyclotomic_numbers_sum}: Every cyclotomic number of order $e$ that equals $0$ splits into four cyclotomic numbers of order $2e$ that equal~$0$. Every cyclotomic number of order $e$ that takes the value $1$ splits into three cyclotomic numbers of order $2e$ that equal $0$ and one cyclotomic number of order $2e$ that equals $1$. If $p^r-1 \equiv 0 \pmod{4}$, the unique cyclotomic number of order $e$ that equals $p^r-2$ splits into one cyclotomic number of order~$2e$ that equals $\frac{1}{4}(p^r-5)$ and three cyclotomic numbers of order~$2e$ that equal $\frac{1}{4}(p^r-1)$. If $p^r-1 \equiv 2 \pmod{4}$, then we obtain $\frac{1}{4}(p^r-3)$ three times and $\frac{1}{4}(p^r+1)$ once.\par
	Denote by $n_e(N)$ the number of cyclotomic numbers of order $e$ that equal~$N$. These numbers were given in \cref{eq:Ne}. By the above argumentation, we obtain the following values for~$n_{2e}(N)$. If $p^r-1 \equiv 0 \pmod{4}$, then
	\begin{align}
		\label{eq:N2e0}
		n_{2e}(0) &= 4n_e(0) + 3 n_e(1),& n_{2e}(1) &= n_e(1) \nonumber\\
		n_{2e}((p^r-5)/4) &= n_e(p^r-2), & n_{2e}((p^r-1)/4) &= 3n_{e}(p^r-2).
	\end{align}
	If $p^r-1 \equiv 2 \pmod{4}$, then
	\begin{align}
		\label{eq:N2e2}
		n_{2e}(0) &= 4n_e(0) + 3 n_e(1),& n_{2e}(1) &= n_e(1) \nonumber\\
		n_{2e}((p^r-3)/4) &= 3n_e(p^r-2), & n_{2e}((p^r+1)/4) &= n_{e}(p^r-2).
	\end{align}	
	From \autoref{prop:cyclo_BIN}, it follows that we need to multiply these numbers with $\frac{1}{2}p^{2r}(p^{2r}-1)$ to obtain the multiplicities of the respective block intersection numbers. For the block intersection number $0$ we additionally need to add $\frac{1}{2}p^{2r}(2p^{2r}+2)(2p^{2r}+1)$.
\end{proof}

Next, we examine the intersection numbers of $dev(E^H)$, the design associated to the disjoint difference family $E^H$ in the Galois ring $\GR(p^2,r)$ from \autoref{th:squares_GR_DDF}. Since the design $dev(E^H)$ is constructed by letting the additive group of $\GR(p^2,r)$ act on the difference family $E^H$, there is a strong connection between differences and block intersection numbers: Let $E^H_i, E^H_j \in E^H$ be two distinct base blocks of $dev(E^H)$, and let~$d$~be a difference occurring $N_d$ times in the multiset $E^H_i - E^H_j$. Then, $N_d$ is the block intersection number $|E^H_i \cap (E^H_j+d)|$ of the blocks $E_i^H$ and $E_j^H+d$ of $dev(E^H)$. Hence, to calculate the block intersection number $|E^H_i \cap (E^H_j+d)|$ we need to calculate the multiplicity $N_d$ of $d$ in $E^H_i - E^H_j$. We will do exactly this for certain base blocks of $dev(E^H)$.\par
Let $\xi$ be a generator of the Teichmüller group $\T^*$ and let $\T = \T^* \cap \{0\}$. As in \autoref{th:squares_GR_DDF}, we denote by~$\T_S^*$ the subgroup of Teichmüller squares and by $\T_N^*$ the set of Teichmüller non-squares. Furthermore, we call a coset of type 
\[
	(1+p\alpha)\T_S^*, 
\]
where $\alpha \in \T$, a \emph{square coset of $\T_S^*$}, and a coset of type
\[
	(1+p\alpha) \T_N^* = (1+p\alpha)\xi\T_S^*,
\] 
where $\alpha \in \T$, a \emph{non-square coset of $\T_S^*$}. In the remaining part of this section, we will establish bounds on block intersection numbers of $dev(E^H)$ that come from the multisets $\Delta \T_S^*$ and $\T_S^*-\T^*_N$. We begin by analyzing the structure of these multisets.

\begin{lemma}
	\label{lem:number_of_SN_cosets}
	Let $p$ be an odd prime. Using the same notation as above, consider the multisets $\Delta \T_S^*$ and $\T_S^* - \T_N^*$ in the Galois ring $\GR(p^2,r)$.
	\begin{itemize}
		\item If $p^r-1 \equiv 0 \pmod{4}$, then $\Delta \T_S^*$ contains $\frac{p^r-5}{4}$ square cosets and $\frac{p^r-1}{4}$~non-square cosets of $\T_S^*$, and $\T_S^*-\T_N^*$ contains $\frac{p^r-1}{4}$ square and non-square cosets of $\T_S^*$ each.
		\item If $p^r-1 \equiv 2 \pmod{4}$, then $\Delta \T_S^*$ contains $\frac{p^r-3}{4}$ square and non-square cosets of $\T_S^*$ each, and $\T_S^*-\T_N^*$ contains $\frac{p^r-3}{4}$ square cosets and $\frac{p^r+1}{4}$ non-square cosets of~$\T_S^*$.			
	\end{itemize}
\end{lemma}
\begin{proof}
	Denote by $\I$ the maximal ideal of $\GR(p^2,r)$. The Teichmüller set $\T$ is a system of representatives of $\GR(p^2,r) / \I$ which is isomorphic to the finite field $\F_{p^{r}}$. Hence, the sets of Teichmüller squares~$\T_S^*$ and Teichmüller non-squares~$\T_N^*$ act in the same way as the respective sets of squares and non-squares in $\F_{p^r}$. The result now follows from \autoref{lem:squares+1}: Let $d = s - s'$ be the difference of two distinct nonzero squares $s,s'$ in $\F_{p^r}$. Equivalently, 
	\[
	sd^{-1}-s'd^{-1} = 1.
	\]
	Note that $d^{-1}$ is a square if and only if $d$ is a square. Using the notation from \autoref{lem:squares+1}, the equation $s-s' = d$ has $QQ$ solutions for $s,s'$ if $d$ is a square, and $NN$ solutions if $d$ is a non-square. Analogously, we obtain the number of solutions for $s,n$ of $s-n=d$, where $s$ is a nonzero square and $n$ is a non-square in $\F_{p^r}$.
\end{proof}

Furthermore, we need the following properties of squares and non-squares in the Galois ring~$\GR(p^2,r)$.
\begin{proposition}
	\label{prop:S_N_properties}
	Consider the Galois ring $\GR(p^2,r)$, where $p$ is odd. Denote by $\T_S^*$ the set of Teichmüller squares and by $\T_N^*$ the set of Teichmüller non-squares.
	\begin{enumerate}
		\item If $p^r-1 \equiv 0 \pmod{4}$, then $-1$ is a Teichmüller square, and $\T_S^* = -\T_S^*$ and $2\T_S^* \subseteq \Delta \T_S^*$.\par
		If $p^r-1 \equiv 2 \pmod{4}$, then $-1$ is a Teichmüller non-square, and $\T_S^* = -\T_N^*$ and $2\T_S^* \subseteq \T_S^*-\T_N^*$.
		\item If $p^r-1 \equiv 0 \pmod{12}$, then $1 \in \Delta \T_S^*$, and $\T_S^* \subseteq \Delta \T_S^*$.\par
		If $p^r-1 \equiv 6 \pmod{12}$, then $1 \in \T_N^* - \T_S^*$, and $\T_S^* \subseteq \T_N^*-\T_S^*$.
		\item If $p^r-1 \equiv 0 \mbox{ or } 6 \pmod{8}$, then $2$ is a square, and $2\T_S^*$ is a square coset of~$\T_S^*$.\par
		If $p^r-1 \equiv 2 \mbox{ or } 4 \pmod{8}$, then $2$ is a non-square, and $2\T_S^*$ is a non-square coset of $\T_S^*$.
	\end{enumerate}
\end{proposition}
\begin{proof}
	Let $\xi$ be a generator of the Teichmüller group $\T^*$ in the Galois ring~$\GR(p^2,r)$.
	\begin{enumerate}
		\item The present authors \cite{kaspers2019} proved that if $p$ is odd, $-1$ is contained in $\T^*$, in particular $-1 = \xi^{\frac{1}{2}(p^r-1)}$. The exponent $\frac{1}{2}(p^r-1)$ is even if $p^r-1 \equiv 0 \pmod{4}$, then $-1$ is a square in $\T^*$. If $p^r-1 \equiv 2 \pmod 4$, the exponent $\frac{1}{2}(p^r-1)$ is odd, hence $-1$ is a non-square in $\T^*$.
		
		\item If $p^r-1 \equiv 0 \pmod{6}$, the equation $x^6 = 1$ has exactly six solutions in the Teichmüller group $\T^*$, namely $\xi^{k(p^r-1)/6}$, where $k \in \{0,1,\dots,5\}.$ We show that the sum of these elements is $0$. It is easy to see that
		\[
			\xi^{(p^r-1)/6} \sum_{k=0}^{5} \xi^{k(p^r-1)/6} = \sum_{k=0}^{5} \xi^{k(p^r-1)/6}.
		\]
		Hence,
		\[
			\left(\xi^{(p^r-1)/6} - 1\right)\sum_{k=0}^{5} \xi^{k(p^r-1)/6}   = 0.
		\]
		As we have shown in the proof of \autoref{th:squares_GR_DDF}, the element $\xi^{(p^r-1)/6} - 1$ is a unit. It follows that
		\begin{align}
		\label{eq:rootsofunity}
			\sum_{k=0}^{5} \xi^{k(p^r-1)/6}   = 0.
		\end{align}
		By the same reasoning, $\sum_{k=0}^{2} \xi^{k(p^r-1)/3}   = 0$. Consequently, we can rewrite \cref{eq:rootsofunity} as 
		\[
			\xi^{5(p^r-1)/6}-\xi^{2(p^r-1)/3} = 1.
		\]
		If $p^r-1 \equiv 0 \pmod {12}$, the elements $\xi^{5(p^r-1)/6}$ and $\xi^{2(p^r-1)/3}$ are squares and, consequently, $1 \in \Delta \T_S^*$. If $p^r-1 \equiv 6 \pmod {12}$, then $\xi^{5(p^r-1)/6}$ is a non-square and $\xi^{2(p^r-1)/3}$ is a square, hence $1 \in \T_N^*-\T_S^*$.
		
		\item We first consider $r=1$. Note that $\GR(p^2,1) = \Z_{p^2}$. The following classical results about quadratic residues were first systematically given by \textcite{gauss1981}. An element~$a$ relatively prime to an odd prime $p$ is a square in $\Z_{p^m}$ if and only if $a$ is a square in~$\Z_p$. In~$\Z_p$, the element~$2$ is a square if $p-1 \equiv 0$ or $6 \pmod{8}$, and $2$ is a non-square if $p-1 \equiv 2$ or $4 \pmod{8}$. This solves the problem for $r=1$.\par
		
		Now, let $r \ge 2$. Let 
		\[
		\T_1^* = \{1, \zeta,\zeta^2, \dots, \zeta^{p-2}\}
		\]
		denote the Teichmüller group of $\GR(p^2,1)$, and let $\T_1 = \T_1^* \cup \{0\}$. For a fixed prime~$p$, the Galois ring $\GR(p^2,1)$ is a subring of $\GR(p^2,r)$ for all~$r \ge 1$. If $\T^* = \langle \xi \rangle$ denotes the Teichmüller group of $\GR(p^2,r)$, then $\T_1^*$ is a subgroup of $\T^*$ and we write
		\[
		\T_1^* = \left\lbrace 1, \xi^{(p^r-1)/(p-1)}, \xi^{2(p^r-1)/(p-1)}, \dots, \xi^{(p-2)(p^r-1)/(p-1)}\right\rbrace,
		\]
		where $\zeta^k = \xi^{k(p^r-1)/(p-1)}$. Since $2$ is a unit in $\GR(p^2,1)$, we write $2 = (1+p\alpha_0)\alpha_1$ for unique~$\alpha_0, \alpha_1$, where $\alpha_0 \in \T_1$ and $\alpha_1 \in \T_1^* $. It follows that $\alpha_1 = \zeta^\ell$ for some $\ell \in \{0,1,\dots,p-2\}$. In $\GR(p^2,r)$, we consequently obtain
		\[
		2= \left(1+p\alpha_0\right)\xi^{\ell(p^r-1)/(p-1)}.
		\]
		Hence, $2$ is a square, and thereby $2\T_S^*$ is a square coset of $\T_S^*$, if at least one of the two numbers $\ell$ and $(p^r-1)/(p-1)$ is even. The second number is even if and only if $r$ is even. In this case, $p^r -1 \equiv 0 \pmod 8$. Hence, if $r$ is odd, the number $\ell$ needs to be even. This is the case if and only if~$2$ is a square in $\GR(p^2,1)$, which, according to the case $r=1$, holds whenever $p-1 \equiv 0$ or~$6 \pmod{8}$. If~$r$ is odd, $p^r \equiv p \pmod{8}$. The result follows.
	\end{enumerate}
	\vspace{-2.2em}
\end{proof}	

By combining all three results from \autoref{prop:S_N_properties}, we obtain the following corollary:

\begin{corollary}
	\label{cor:mod24}
	Consider the Galois ring $\GR(p^2,r)$, where $p$ is odd. Denote by $\T_S^*$ the set of Teichmüller squares and by $\T_N^*$ the set of Teichmüller non-squares.
	\begin{itemize}
		\item If $p^r-1 \equiv 0 \pmod{12}$, then the multiset $\Delta \T_S^*$ contains both $\T_S^*$ and $2\T_S^*$, and the set $2\T_S^*$ is a square coset of $\T_S^*$ if and only if $p^r-1 \equiv 0 \pmod{24}$.
		\item If $p^r-1 \equiv 6 \pmod{12}$, then the multiset $\T_S^*-\T_N^*$ contains both $\T_N^*$ and $2\T_S^*$, and the set $2\T_S^*$ is a non-square coset of $\T_S^*$ if and only if $p^r-1 \equiv 18 \pmod{24}$.
	\end{itemize}
\end{corollary}

Note that $p^r-1 \equiv 0 \pmod{24}$ holds whenever the prime $p \ge 5$ and $r$ is even. To continue, we need the following result about when $2$ is a Teichmüller square.
\begin{lemma}
	\label{lem:Ts=2Ts}
	Consider the Galois ring $\GR(p^2,r)$, where $p$ is odd. Denote by $\T_S^*$ the set of Teichmüller squares and by $\T_N^*$ the set of Teichmüller non-squares.
	\begin{itemize}
		\item If both $p^r-1 \equiv 0$ or $6 \pmod{8}$ and $2^{p-1} \equiv 1 \pmod{p^2}$, then $\T_S^* = 2\T_S^*$.
		\item If both $p^r-1 \equiv 2$ or $4 \pmod {8}$ and $2^{p-1} \equiv 1 \pmod{p^2}$, then $\T_N^* = 2\T_S^*$.
	\end{itemize}
\end{lemma}
\begin{proof}
	The equation $\T_S^* = 2\T_S^*$ holds if and only if $2$ is a square in the Teichmüller group~$\T^*$. According to \autoref{prop:S_N_properties}, the element $2$ is a square in $\GR(p^2,r)$ if and only if $p^r-1 \equiv 0$ or $6 \pmod{8}$. Since $\T^*$ has order $p^r-1$, the element $2$ is contained in $\T^*$ if and only if $2^{p^r-1} \equiv 1 \pmod{p^2}$. Since $2$ is also an element of $\Z_{p^2} = \GR(p^2,1)$ which is a subring of $\GR(p^2,r)$, this condition can be reduced to $2^{p-1} \equiv 1 \pmod{p^2}$.\par
	On the other hand, the equation $\T_N^* = 2\T_S^*$ holds if and only if $2$ is a non-square and $2 \in \T^*$. The second statement now follows by analogous reasoning as above from \autoref{prop:S_N_properties}.
\end{proof}

Primes that solve $2^{p-1} \equiv 1 \pmod{p^2}$ are called \emph{Wieferich primes}. So far, the only known Wieferich primes are $1093$ and $3511$. Thus, the only known Galois rings of characteristic~$p^2$ satisfying $\T_S^* = 2 \T_S^*$ are $\GR(1093^2,r)$, where $r$ is even, and $\GR(3511^2,r)$ for arbitrary $r$. The only known Galois ring of characteristic~$p^2$ satisfying $\T_N^* = 2 \T_S^*$ is $\GR(1093^2,r)$, where $r$ is odd.

With the help of the first result given in \autoref{prop:S_N_properties}, we now establish a lower bound on the multiplicities of certain differences of Teichmüller elements. This is an analogue of the present authors' previous result \cite[Lemma~5.9]{kaspers2019} for the design $dev(C)$ from \autoref{th:Davis_EDF}:

\begin{lemma}
	\label{lem:lowerbound}
	Consider the Galois ring $\GR(p^2,r)$, where $p$ is odd. Denote by $\T_S^*$ the set of Teichmüller squares and by $\T_N^*$ the set of Teichmüller non-squares.
	\begin{itemize}
		\item If $p^r-1 \equiv 0 \pmod{4}$, then all differences~$d \in \Delta \T_S^*$ where $d \notin 2\T_S^*$ have multiplicity $N_d > 1$ in $\Delta \T_S^*$. 
		\item If $p^r-1 \equiv 2 \pmod{4}$, then all differences~$d \in \T_S^*-\T_N^*$ where $d \notin 2\T_S^*$ have mulitplicity $N_d > 1$ in $\T_S^* - \T_N^*$. 
	\end{itemize}
\end{lemma}
\begin{proof}
	We prove the first result. The proof of the second statement is analogous. Let~$p$ be a prime and $r$ be a positive integer such that $p^r-1 \equiv 0 \pmod{4}$. Moreover, let $d \in \Delta \T_S^*$, which means that $d = s-s'$ is the difference of two distinct Teichmüller squares~$s,s' \in \T_S^*$. According to \autoref{prop:S_N_properties}, $\T_S^* = -\T_S^*$. Hence, if $s' \ne s$, then $(-s') -(-s) = d$ is a second representation of $d$ in $\Delta \T_S^*$. Note that all these differences occur in pairs. If $s' = -s$, however, the two representations are the same, and $d = 2s$, thus $d \in 2\T_S^*$. The statement follows.
\end{proof}

In the following lemma, we will establish an upper bound on the multiplicity of certain differences in $\Delta \T_S^*$ and $\T_S^* - \T_N^*$. For our main theorem, only the first part of the lemma is relevant. However, we also state the second part as it is easily obtained from the previous results.

\begin{lemma}
	\label{lem:upperbound}
	Let $p$ be an odd prime such that $2^{p-1} \not\equiv 1 \pmod{p^2}$. Consider the Galois ring $\GR(p^2,r)$, and denote by $\T_S^*$ the set of Teichmüller squares and by $\T_N^*$ the set of Teichmüller non-squares.
	\begin{itemize}
		\item If $p^r-1 \equiv 0 \pmod{24}$, then all differences $d \in \Delta \T_S^*$ where $d$ is a square have multiplicity $N_d < \frac{p^r-5}{4}$ in $\Delta \T_S^*$
		\item If $p^r-1 \equiv 18 \pmod{24}$, then all differences $d \in \T_S^*-\T_N^*$ have multiplicity $N_d < \frac{p^r+1}{4}$ in $\T_S^* - \T_N^*$.
	\end{itemize} 
\end{lemma}
\begin{proof}
	The condition $2^{p-1} \not\equiv 1 \pmod{p^2}$ ensures that $2 \notin \T^*$, and consequently $\T_S^* \ne 2\T_S^* \ne \T_N^*$ as we showed in \autoref{lem:Ts=2Ts}. Assume $p^r-1 \equiv 0 \pmod{24}$, and let $d$ be the difference of two Teichmüller squares, $d \in \Delta \T_S^*$, such that~$d$ is a square. Denote by $N_d$ the multiplicity of $d$ in $\Delta \T_S^*$. From \autoref{lem:number_of_SN_cosets}, we know that $\Delta \T_S^*$ contains $\frac{1}{4}(p^r-5)$ not necessarily distinct square cosets of~$\T_S^*$. It follows that $N_d \le \frac{1}{4}(p^r-5)$, and $N_d = \frac{1}{4}(p^r-5)$ if and only if $\Delta \T_S^*$ contains only exactly one square coset of $\T_S^*$ with multiplicity $\frac{1}{4}(p^r-5)$. Assume $N_d = \frac{1}{4}(p^r-5)$. If $p^r-1 \equiv 0 \pmod{24}$, then, according to \autoref{cor:mod24}, both $\T_S^*$ and $2\T_S^*$ are subsets of $\Delta \T_S^*$, and $2\T_S^*$ is a square coset of $\T_S^*$. This is a contradiction.\par 
	
	Now, assume $p^r-1 \equiv 18 \pmod{24}$, and let $d$ be the difference of a Teichmüller square and a Teichmüller non-square, $d \in \T_S^* - \T_N^*$, such that $d$ is a non-square. Denote by $N_d$ the multiplicity of $d$ in $\T_S^* - \T_N^*$. Analogously to above, we conclude from \autoref{lem:number_of_SN_cosets} that $N_d \le \frac{1}{4}(p^r+1)$, and $N_d = \frac{1}{4}(p^r+1)$ if and only if $\T_S^* - \T_N^*$ contains only exactly one non-square coset of $\T_S^*$. Assume $N_d = \frac{1}{4}(p^r+1)$. If $p^r-1 \equiv 18 \pmod{24}$, then both $\T_N^*$ and $2\T_S^*$ are contained in $\T_S^* - \T_N^*$, and $2\T_S^*$ is a non-square coset of $\T_S^*$. Again, we obtain a contradiction.		
\end{proof}

As we have mentioned above, the multiplicity of a difference $d$ in $\Delta \T_S$ corresponds directly to the block intersection number $|\T_S^* \cap (\T_S^* \cap d)|$. Hence, we obtain from the previous lemmas the following theorem which is our main theorem.

\begin{theorem}
	\label{th:squares_noniso}
	Let $p$ be an odd prime such that $2^{p-1} \not\equiv 1 \pmod{p^2}$. Let $C^H$ be a $\left(p^{2r}, \frac{p^r-1}{2}, \frac{p^r-3}{2}\right)$ disjoint difference family in the additive group of the finite field $\F_{p^{2r}}$ constructed with \autoref{th:Fq_DDF}, and let $E^H$ be a disjoint difference family with the same parameters in the additive group of the Galois ring $\GR(p^2,r)$ constructed with \autoref{th:squares_GR_DDF}. If $p^r-1 \equiv 0 \pmod {24}$, the $2$-$\left(p^{2r}, \frac{p^r-1}{2}, \frac{p^r-3}{2}\right)$ designs $dev(E^H)$ and $dev(C^H)$ are nonisomorphic.
\end{theorem}
\begin{proof}
	Let $p$ be an odd prime and $r$ be an integer such that $p^r-1 \equiv 0 \pmod{24}$. Recall from \autoref{th:intnums_Fp2r} that in this case the block intersection numbers of our design $dev(C^H)$ are given as $0,1,\frac{1}{4}(p^r-5), \frac{1}{4}(p^r-1)$. Now, consider the Galois ring $\GR(p^2,r)$, and denote by $\T_S^*$ the set of Teichmüller squares and by $\T_N^*$ the set of Teichmüller non-squares. By combining \autoref{lem:lowerbound} and \autoref{lem:upperbound}, we obtain
	\[
	1 < \left|\T_S^* \cap \left(\T_S^* + d \right) \right| < \frac{p^r-5}{4}
	\]
	for all squares $d \in \Delta \T_S^* \setminus 2\T_S^*$. Consequently, the design $dev(E^H)$ has an intersection number different from the ones of $dev(C^H)$, and the designs are nonisomorphic.
\end{proof}

We remark that $p^r-1 \equiv 0 \pmod{24}$ holds for all $p$ and $r$ where $p\ge5$ and $r$ is even. Furthermore, we remark that for the Wieferich primes $1093$ and $3511$, that satisfy $2^{p-1}\equiv 1 \pmod{p^2}$, the designs $dev(C^H)$ and $dev(E^H)$ are nonisomorphic for all~$r > 1$.  With the help of \texttt{Magma}~\cite{magma} we computed the multisets $\Delta \T_S^*$ and $\T_S^*-\T_N^*$ for $r=1$ and checked that these multisets contain more than one square coset and more than two non-square coset of~$\T_S^*$ each. Consequently, there will be at least as many square and non-square cosets of $\T_S^*$ in $\Delta \T_S^*$ and $\T_S^*-\T_N^*$, respectively, for~$r > 1$. So, the bounds on the respective block intersection numbers established in the previous proof hold.\par

To conclude this section, we give an example that demonstrates why our block intersection number approach fails if $p^r-1 \not \equiv 0 \pmod{24}$. We choose as an example the case $p^r-1 \equiv 18 \pmod{24}$ since, in the previous lemmas, we have already obtained several results about this case that followed immediately from the results for $p^r-1 \not \equiv 0 \pmod{24}$.

\begin{example}
	\label{ex:fail}
	Let $p$ and $r$ such that $p^r-1 \equiv 18 \pmod{24}$. In this case, according to \autoref{th:intnums_Fp2r}, the design~$dev(C^H)$, has block intersection numbers $0$, $1$, $\frac{1}{4}(p^r-3)$, $\frac{1}{4}(p^r+1)$. For the design $dev(E^H)$, using \autoref{lem:lowerbound} and \autoref{lem:upperbound}, we obtain
	\[
		1 < \left|\T_S^* \cap \left(\T_N^* + d \right) \right| < \frac{p^r+1}{4}
	\]
	for all $d \in (\T_S^*-\T_N^*) \setminus 2\T_S^*$. However, this result is of little use as it is still possible that there exists $d \in (\T_S^*-\T_N^*) \setminus 2\T_S^*$ such that $\left|\T_S^* \cap \left(\T_N^* + d \right)\right| = \frac{1}{4}(p^r-3)$ or that two completely different blocks intersect in $\frac{1}{4}(p^r+1)$ elements. In fact, the multiset $p\T_S^* - p\T_N^*$ contains $\frac{1}{4}(p^r+1)$ times the set $p\T_N^*$ and $\frac{1}{4}(p^r-3)$ times the set $p\T_S^*$. Hence, these two numbers actually occur as the block intersection numbers 
	\[
		\left|p\T_S^* \cap \left(p\T_N^* + d \right) \right|,
	\]
	where $d \in \I \setminus\{0\}$. Hence, we cannot show the existence of an intersection number $N$ such that $1 < N < \frac{1}{4}(p^r-3)$.
\end{example}

\section{Conclusion and open questions}
\label{sec:conclusion}
Motivated by the present authors' \cite{kaspers2019} recent results, we tried to use the same technique to solve another isomorphism problem about $2$-$(v,k,k-1)$ designs. Thanks to the algebraic structure of our designs, we were able to solve the problem for many cases, and, in doing so, obtained some interesting results about cyclotomic numbers and the structure of Galois rings of characteristic $p^2$. But the isomorphism problem is still not solved for all cases. \par
Our results demonstrate that using block intersection numbers as a method to tackle isomorphism problems about combinatorial designs has its limitations. One needs designs that have a sufficiently strong algebraic structure to calculate or even bound these numbers. We still consider this approach promising, especially if the designs are constructed as the developments of some difference structures. During our studies, we discovered the following interesting open problems:\par
\begin{itemize}
	\item Our computations hint that \autoref{th:squares_noniso} holds for all $p$ and $r$, where $p$ is odd. However, our examination of intersection numbers did not lead to the results necessary to prove this conjecture. We leave this task to future work.
	\item The construction of a disjoint difference family in $\GR(p^2,r)$ presented in \autoref{th:squares_GR_DDF} does not only work for the subgroup of squares in the Teichmüller group but for all its subgroups. Moreover, there will always be an analogue in $\F_{p^{2r}}$. It would be interesting to study the isomorphism problem for the associated designs in all these cases. It might be possible to deduce more block intersection numbers from the ones given in this paper and in \cite{kaspers2019}.
	\item As mentioned before, nonisomorphic designs can have the same block intersection numbers. It would be interesting to find more difference families as in \autoref{rem:feng} for which their associated designs have the same intersection numbers but are still nonisomorphic.
\end{itemize}

\printbibliography
\end{document}